\documentclass[11pt,reqno]{amsart}
\pdfoutput=1
\usepackage{amssymb,latexsym}
\usepackage{graphicx}

\usepackage[usenames]{xcolor}
\usepackage{amssymb}
\usepackage{mathtools}
\usepackage{amsthm}
\usepackage{amsfonts}
\usepackage{amscd}
\usepackage{graphicx}

\usepackage{fullpage}
\usepackage{float}

\usepackage[colorlinks=true,
linkcolor=webgreen,
citecolor=webgreen]{hyperref}

\definecolor{webgreen}{rgb}{0,.5,0}

\usepackage{booktabs}
\usepackage{enumitem}

\usepackage{tikz}
\usetikzlibrary{matrix}

\calclayout

\theoremstyle{plain}
\numberwithin{equation}{section}
\newtheorem{thm}{Theorem}[section]
\newtheorem{theorem}[thm]{Theorem}
\newtheorem{lemma}[thm]{Lemma}
\newtheorem{proposition}[thm]{Proposition}
\newtheorem{conjecture}[thm]{Conjecture}

\theoremstyle{definition}
\newtheorem{definition}[thm]{Definition}
\newtheorem{example}[thm]{Example}

\def\modd#1 #2{#1\ \mbox{\rm (mod}\ #2\mbox{\rm )}}

\begin{document}

\title{A Refinement of the $3x+1$ Conjecture}
\author{Roger Zarnowski}
\address{Department of Mathematics and Statistics\\
                Northern Kentucky University\\
                Highland Heights, KY 41017\\
                USA}
\email{zarnowskir1@nku.edu}
\date{\today}

\begin{abstract}
We reformulate the $3x+1$ conjecture by restricting attention to numbers congruent to $2$ (mod $3$). This leads to an equivalent conjecture for positive integers that reveals new aspects of the dynamics of the $3x+1$ problem.  Advantages include a governing function with particularly simple mapping properties in terms of partitions of the set of integers. We use the refined conjecture to obtain a new characterization of $3x+1$ trajectories that shows a special role played by numbers congruent to $2$ or $8$ (mod $9$).  We construct an accelerated iteration whose long-term behavior involves only those numbers.
\end{abstract}

\maketitle

\section{Introduction}
The $3x+1$ problem pertains to iteration of the following function defined on the set $\mathbb{Z}$ of integers:

\begin{displaymath}
T(n)=
\begin{dcases}
  \hspace{0.45cm} \dfrac{n}{2}, & \text{if\ } n\equiv 0 \pmod 2, \\
  \dfrac{3n+1}{2},   & \text{if\ } n\equiv 1 \pmod 2.
\end{dcases}
\end{displaymath}

We adopt the notation $T^{\,0}(n)=n$ and $T^{\,k+1}(n)=T(T^{\,k}(n))$ for $k=0, 1, 2, \ldots$.  We refer to the sequence of iterates
$(T^{\,k}(n))_{k\geq 0}$ as the $T$-trajectory of $n$, with similar terminology for  other functions.

The $3x+1$ conjecture
states that, for every positive integer $n$, the $T$-trajectory of $n$ eventually reaches the cycle
$(2,1)$.  The $3x+1$ problem, to either prove or disprove this conjecture, remains unsolved after decades of attention and continues to be a problem of great interest.
Lagarias \cite{Lag1} provides a comprehensive survey of the problem, with  extensive references and full text of some of the historically most important related articles.

In what follows, we show that the standard formulation of the conjecture in terms of $T$ involves extraneous features that obscure important  aspects of the problem.  We present a modified form of the conjecture that eliminates these unnecessary features and reveals interesting structure in the underlying dynamics.  We derive the refined conjecture and the associated replacement for $T$ in Section \ref{SecRefining}, following some preliminary observations. In Section \ref{SecFeatures} we explore advantages of the new form of the conjecture.  An immediate consequence is that the refined conjecture leads to shorter, smoother trajectories. More significantly, however, the modified function has special mapping properties that are particularly easy to describe and which have no equivalent in the traditional form of the conjecture.  We use those properties in Section \ref{SecReversion} to obtain a new characterization of $T$-trajectories, showing that trajectories are dominated by numbers congruent to $2$ or $8$ (mod $9$).  We then construct an accelerated iteration that effectively involves only those numbers. In Section \ref{SecComp} we present some computational observations.  We close in Section \ref{SecSummary} with ideas for further exploration.

\section{Refining the \texorpdfstring{$3x+1$}{3x+1} iteration}\label{SecRefining}

We begin with some simple facts that follow from the definition of $T$.

\begin{proposition}\label{PropA}
The $T$-trajectory of any nonzero integer $n$ consists of finitely many numbers congruent to $\modd 0 3$ followed only by numbers congruent to $1$ or $\modd 2 3$.
\end{proposition}

\begin{proof}
We may write any nonzero $n$ in the form $n=2^k m$, with $k\geq 0$ and $m$ odd.  The first $k$ iterates of $n$ are then obtained by divisions by $2$, so these numbers are either all congruent to $0$ (mod $3$) or all not congruent to $0$ (mod $3$).  The $k^{th}$ iterate is $m$, which is odd.  But no iterate of an odd number is divisible by $3$, since $3n+1$ is never divisible by $3$ and such divisibility is not affected by subsequent divisions by $2$.
\end{proof}

Numbers divisible by $3$ are therefore transient elements of the  $T$-trajectory of any nonzero $n$, appearing at most as initial terms of such a trajectory. For purposes of exploring the long-term behavior of trajectories, we may then restrict our attention to numbers congruent to $1$ or $2$ (mod $3$).

\begin{proposition}\label{PropB}
If $n\equiv \modd 1 3$, then $T(n)\equiv \modd 2 3$.
\end{proposition}
\begin{proof}
If $n\equiv 1$ (mod $3$), then $n$ is of the form $6j+1$ or $6j+4$. But $T(6j+1)=9j+2$ and $T(6j+4)=3j+2$, each of which is congruent to $2$ (mod $3$).
\end{proof}

We refer to numbers congruent to $1$ (mod $3$) as \emph{isolated} since no two such numbers appear consecutively in a $T$-trajectory.  This is a term we will use again in subsequent sections.

Proposition \ref{PropB} suggests that, in addition to disregarding the transient numbers divisible by $3$, further simplification may be obtained by advancing the iteration one step past numbers congruent to $1$ (mod $3$).  This requires the additional observation that the only predecessor of $3n+1$ in a $T$-trajectory is the even number $6n+2$, since any odd number $2j+1$ iterates to $3j+2 \not\in \{3n+1\}$.  For brevity, we also introduce the notation $[j]_m$ to denote the set of integers congruent to $j$ (mod $m$).

\begin{definition}\label{DefA}
Let $F:[2]_3\rightarrow [2]_3$ be defined by
\begin{equation}\label{EqnF}
F(n)=
\begin{dcases}
T(n), & \text{\ if \ } n\equiv 5 \pmod 6 \\
T^{\,2}(n), & \text{\ if \ } n\equiv 2 \pmod 6.
\end{dcases}
\end{equation}
\end{definition}
For $n\in [2]_3$, the $F$-trajectory of $n$ is then the same as the $T$-trajectory of $n$ with numbers congruent to $1$ (mod $3$) removed.

\begin{lemma}\label{LemA}
The $3x+1$ conjecture is true if and only if the $F$-trajectory of any positive integer in $[2]_3$ converges to $2$.
\end{lemma}
\begin{proof}
Suppose the $3x+1$ conjecture holds, so that the $T$-trajectory of any positive $n\in\left[2\right]_3$ eventually reaches $2$.   The $F$-trajectory of $n$ therefore also reaches $2$, which is a fixed point of $F$.

Conversely, suppose the $F$-trajectory of any positive integer in $\left[2\right]_3$ reaches $2$. If $n$ is an arbitrary positive integer, then by Propositions \ref{PropA} and  \ref{PropB}, there is some $k$ such that $T^k(n)\in\left[2\right]_3$. Since the $F$-trajectory of $T^k(n)$ consists of further iterations by $T$, the $T$-trajectory of $T^k(n)$ eventually reaches $2$, which says that the $T$-trajectory of $n$ reaches $2$.
\end{proof}

Having expressed the conjecture in terms only of integers congruent to $2$ (mod $3$), it is now possible to convert once again to an iteration on all of $\mathbb{Z}$.  As we show below, this leads to the following reformulation of the $3x+1$ conjecture.

\begin{conjecture}[The refined $3x+1$ conjecture] \label{ConjR}
Let $T_R:\mathbb{Z}\rightarrow \mathbb{Z}$ be defined by
\begin{equation} \label{EqnTR}
T_R(n)=
\begin{dcases}
  \hspace{0.29cm} \dfrac{3n}{4}, & \text{if $n\equiv 0 \pmod 4$}, \\
  \hspace{0.08cm} \dfrac{n-2}{4}, & \text{if $n\equiv 2 \pmod 4$}, \\
  \dfrac{3n+1}{2},   & \text{if $n\equiv 1 \pmod 2$ }.
\end{dcases}
\end{equation}
Then for every integer $n\geq 0$, the $T_R$-trajectory of $n$ converges to $0$.
\end{conjecture}

\begin{theorem}\label{ThmA}
The $3x+1$ conjecture is true if and only if the refined $3x+1$ conjecture is true.
\end{theorem}

\begin{proof}
Let $S:\mathbb{Z}\rightarrow [2]_3$ be defined by $S(n)=3n+2$. Then $S$ is bijective and $S^{-1}(n)=\frac{n-2}{3}$. Define $G:\mathbb{Z}\rightarrow \mathbb{Z}$ by  $$G(n)=(S^{-1}\circ F \circ S)(n).$$ Then $F$ and $G$ are topologically conjugate, so the $G$-trajectory of every $n\geq 0$ converges to $0$  if and only if the $F$-trajectory of every positive $S(n)\in \left[2\right]_3$ converges to $S(0)=2$.  By Lemma \ref{LemA}, this, in turn, is equivalent to the $3x+1$ conjecture since every positive element of $[2]_3$ is equal to $S(n)$ for some $n\geq 0$.  We will show that $G=T_R$ as defined above.

First, if $n\equiv 0$ (mod $4$), $n=4j$, then
\begin{align*}
G(n) &= S^{-1}\left(F(3(4j)+2)\right) \\
&= S^{-1}\left(F(12j+2)\right) \\
&= S^{-1}\left(T^{\,2}(12j+2)\right) \\
&= S^{-1}(9j+2) \\
&= 3j \\
&= \frac{3n}{4}.
\end{align*}

Next, if $n\equiv 2$ (mod $4$), $n=4j+2$, then
\begin{align*}
G(n) &= S^{-1}\left(F(3(4j+2)+2)\right) \\
&= S^{-1}\left(F(12j+8)\right) \\
&= S^{-1}\left(T^{\,2}(12j+8)\right) \\
&= S^{-1}(3j+2) \\
&= j \\
&= \frac{n-2}{4}.
\end{align*}

Finally, if $n\equiv 1$ (mod $2$), $n=2j+1$, then
\begin{align*}
G(n) &= S^{-1}\left(F(3(2j+1)+2)\right) \\
&= S^{-1}\left(F(6j+5)\right) \\
&= S^{-1}\left(T(6j+5)\right) \\
&= S^{-1}(9j+8) \\
&= 3j+2 \\
&= \frac{3n+1}{2}.
\end{align*}

\end{proof}

\section{Features of the refined \texorpdfstring{$3x+1$}{3x+1} conjecture}\label{SecFeatures}

By expressing the $3x+1$ conjecture in terms of $T_R$ instead of $T$, we have effectively filtered out from $T$-trajectories the transient numbers divisible by $3$ and the isolated numbers that are congruent to $1$ (mod $3$). The following result for $T_R$ provides an interesting analogue to Proposition \ref{PropB} for $T$.

\begin{proposition} \label{PropC}
If $n\equiv \modd{1} {3}$ , then $T_R(n)\not\equiv \modd{1} {3}$.
\end{proposition}
\begin{proof}
If $n\equiv \modd{1} {3}$, then $n$ has the form $6j+1$, $12j+4$, or $12j+10$.  But $T_R(6j+1)=9j+2 \equiv \modd{2} {3}$, $T_R(12j+4)=9j+3 \equiv \modd{0} {3}$, and $T_R(12j+10)=3j+2\equiv \modd{2} {3}$.
\end{proof}

By Proposition \ref{PropC}, numbers congruent to $1$ (mod $3$) are isolated terms in $T_R$-trajectories just as they are in $T$-trajectories.

We now explore additional features of the refined mapping $T_R$.  In the subsequent section, we will use these results to obtain new information about $T$-trajectories.

\subsection{Smoothing of trajectories}

  From the proof of Theorem \ref{ThmA}, we have
\[
F(n)=\left(S\circ T_R \circ S^{-1}\right)(n), \text{\ for\ } n\equiv \modd 2 3.
\]
By iterating, it follows that
\begin{equation}\label{EqnSofT}
S\left(T_R^k(n)\right) = F^k\left(S(n)\right), \text{\ for\ }k= 0, 1, 2, \ldots, \text{\ and\ } n\in \mathbb{Z}.
\end{equation}
So elements of a $T_R$-trajectory $\left(n, T_R(n), \ldots \right)$ are mapped by $S$ to $\left(S(n), F(S(n)), \ldots \right)$.  By the definition of $F$, these are the elements of the $T$-trajectory of $S(n)$ that are congruent to $2$ (mod $3$).

\begin{example}
Figure \ref{FigA} illustrates graphically, for $n=65$, both the standard $T$-trajectory and the corresponding $F$-trajectory that is related to $T_R$ by conjugation with $S$.  The terms of the $T$-trajectory, with the elements of the  $F$-trajectory in bold, are
\[
\left(T^k(65)\right)_{k>0}=\left( \mathbf{65}, \mathbf{98}, 49, \mathbf{74}, 37, \mathbf{56}, 28, \mathbf{14}, 7, \mathbf{11}, \mathbf{17}, \mathbf{26}, 13, \mathbf{20}, 10, \mathbf{5}, \mathbf{8}, 4, \mathbf{2}, 1,
\mathbf{2}, \ldots\right). \\
\]
The corresponding $T_R$-trajectory is
\[
\left(T_R^k(21)\right)_{k>0} = \left(21, 32, 24, 18, 4, 3, 5, 8, 6, 1, 2, 0, 0, \ldots\right).
\]
The refined $F$-trajectory is shorter and smoother than the classical $T$-trajectory,  eliminating much oscillatory behavior.  The removed numbers are congruent to $1$ (mod $3$) and so are either of the form $6j+1$ with predecessor $12j+2$ and successor $9j+2$, or the form $6j+4$ with predecessor $12j+8$ and successor $3j+2$.  In either case the central number is less than the mean of its predecessor and successor, so the removed portions of the graph of the $T$-trajectory are always below the graph of the refined trajectory.
\begin{figure}[H]
\begin{center}
    \scalebox{0.65}{\includegraphics{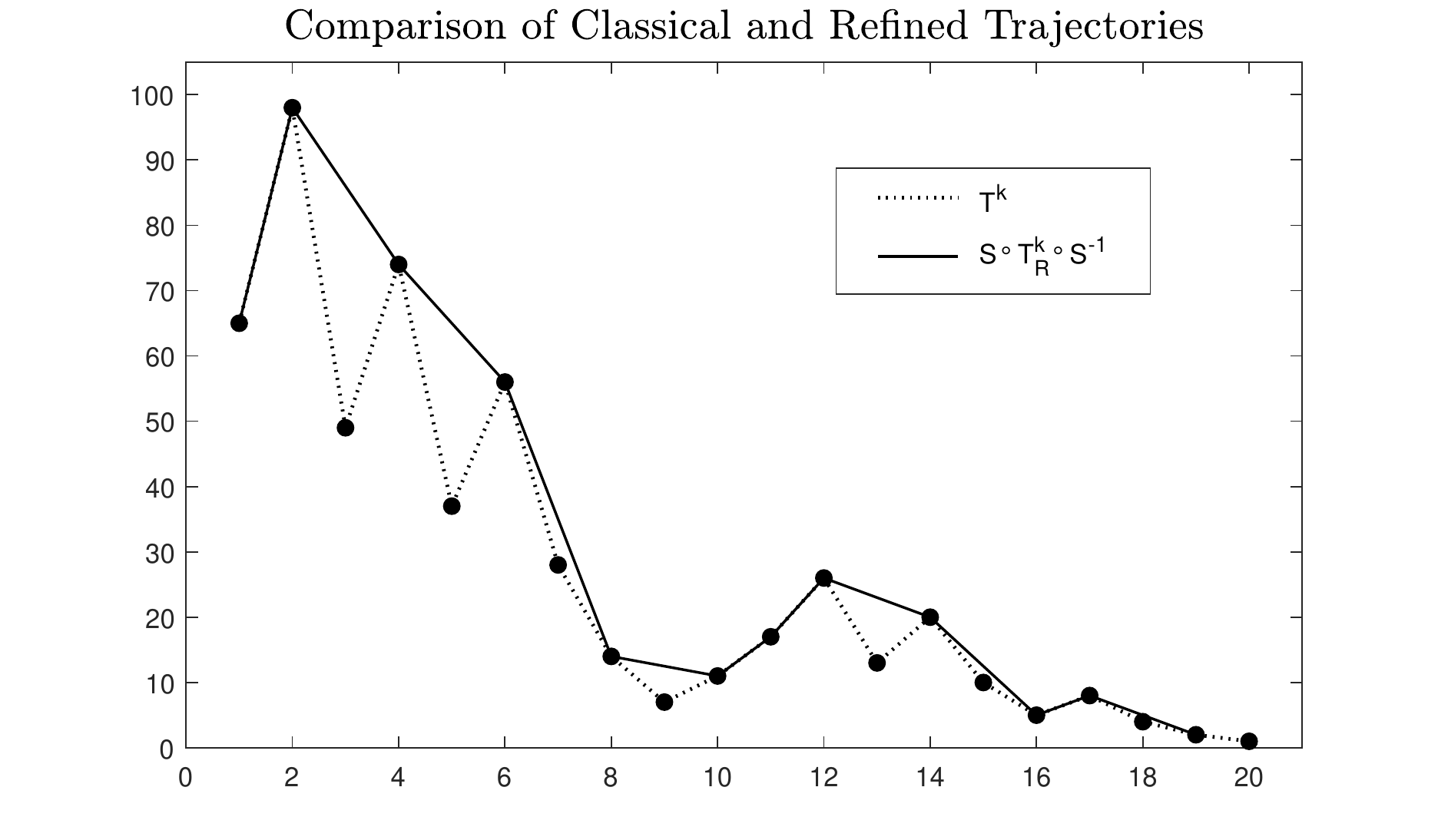}}
\caption{The classical trajectory (dotted) and the refined trajectory (solid) for $n=65$.}
\label{FigA}
\end{center}
\end{figure}
\end{example}

\subsection{Mapping of a 5-partition to a 3-partition}

Consider partitioning the integers into sets of numbers congruent to $1$ (mod $2$),  $0$ (mod $4$) or $2$ (mod $4$).  The latter set can be further partitioned into sets of numbers congruent to $2, 6, \text{\ or \ } 10$ (mod $12$).  This is useful since

\begin{equation}\label{EqnPartitions}
\begin{gathered}
T_R(4j) = T_R(12j+2) = 3j, \\
T_R(12j+6) = 3j+1, \\
T_R(12j+10) = T_R(2j+1) = 3j+2.
\end  {gathered}
\end{equation}
We will find it convenient to rewrite such relationships using the notation $$[b]_a \xrightarrow{f} [d]_c$$ to indicate that $f(aj+b)=cj+d$ for every integer $j$.  The label on the arrow may be dropped when the function $f$ is clear from context.

With this notation, the results expressed in Equations (\ref{EqnPartitions}) are represented in Figure \ref{FigPartitions}. We see that each of five partition subsets of $\mathbb{Z}$ is  mapped element-wise to one of the congruence classes of $3$. This easily described structure holds promise of making the underlying dynamics of the $3x+1$ conjecture amenable to further analysis. We describe one consequence of this structure in the next subsection.

\begin{figure}[H]
\begin{center}
\begin{tikzpicture}
  \matrix (m) [matrix of math nodes,row sep=2em,column sep=-6pt]
  {
     \mathbb{Z} & \hspace{5pt}=\hspace{5pt} & \left[0\right]_4 & \hspace{-5pt}\cup & \hspace{-3pt} \left[2\right]_{12}\ \  & \hspace{-3pt}\cup & \ \ \hspace{-3pt}\left[6\right]_{12} \ \ & \hspace{-3pt}\cup & \ \ \hspace{-3pt}\left[10\right]_{12}\ \ & \hspace{-3pt}\cup & \ \ \hspace{-3pt}\left[1\right]_2 & {}  \\
     {} & {} & {} & \left[0\right]_3 & \ \ \ \ \ \cup & {} & \left[1\right]_3 & {} & \cup \ \ \ \ \ & \hspace{-5pt}\left[2\right]_3 & \hspace{5pt}=\hspace{5pt} & \mathbb{Z} \\};
  \path[-stealth]
    (m-1-3) edge  (m-2-4)
    (m-1-5) edge (m-2-4)
    (m-1-7) edge (m-2-7)
    (m-1-9) edge (m-2-10)
    (m-1-11) edge (m-2-10);
\end{tikzpicture}
\caption{A schematic of the refined $3x+1$ mapping $T_R$.}
\label{FigPartitions}
\end{center}
\end{figure}

\subsection{Symmetries in \texorpdfstring{$T_R$}{TR}-iterates}

The mapping property of $T_R$ as shown in Figure \ref{FigPartitions} suggests that further insights may be obtained by iterating the congruence classes of $3$ to themselves.  We summarize the result of such a process in this section.  Our analysis rests upon the following useful partition:
\begin{align}\label{EqnZpart}
\mathbb{Z} &= [1]_2 \cup [0]_2 \notag\\
    &= [1]_2 \cup [0]_4 \cup [2]_4 \notag\\
    &= [1]_2 \cup [0]_4 \cup [6]_8 \cup [2]_8 \notag\\
    &= [1]_2 \cup [0]_4 \cup [6]_8 \cup [2]_{16} \cup [10]_{16} \notag\\
    &= [1]_2 \cup [0]_4 \cup [6]_8 \cup [2]_{16} \cup [26]_{32} \cup [10]_{32} \notag\\
    &= [1]_2 \cup [0]_4 \cup [6]_8 \cup [2]_{16} \cup [26]_{32} \cup [10]_{64} \cup [42]_{64}.
\end{align}
We now advance the $T_R$ iteration according to the congruence classes in Equation (\ref{EqnZpart}).  This gives an accelerated iteration with special properties.

Define $T_R^*:\mathbb{Z} \rightarrow \mathbb{Z}$ by

\begin{displaymath}
T_R^*(n)=
\begin{dcases}
  T_R(n), & \text{if\ } n \equiv \modd 1 2 \text{\ or\ }  n \equiv \modd 0 4 \\
  T_R^{\,2}(n), & \text{if\ } n \equiv \modd 6 8 \text{\ or\ }  n \equiv \modd 2 {16} \\
  T_R^{\,3}(n), & \text{if\ } n \equiv \modd {26} {32},  n \equiv \modd {10} {64} \text{\ or\ } n \equiv \modd {42} {64}.
\end{dcases}
\end{displaymath}
To describe the mapping properties of $T_R^*$, we first reorder the partition of Equation (\ref{EqnZpart}) as follows:
    \begin{equation*}
    \begin{tikzpicture}
      \matrix (m) [matrix of math nodes,row sep=0em,column sep=-1.5pt]
      {
         \mathbb{Z} & = & \left[0\right]_4 & \cup & \left[2\right]_{16} & \cup &  \left[10\right]_{64} & \cup & \left[42\right]_{64} & \cup & \left[26\right]_{32} & \cup & \left[6\right]_{8} & \cup & \left[1\right]_{2}   \\
      };
    \end{tikzpicture}
    \end{equation*}
We then partition each of these seven subsets into congruence classes (mod $3$), to obtain
    \begin{equation*}
    \begin{tikzpicture}
      \matrix (m) [matrix of math nodes,row sep=0em,column sep=-5pt]
      {
         \left[0\right]_3 & = & \left[0\right]_{12} & \cup & \left[18\right]_{48} & \cup & \left[138\right]_{192} & \cup & \left[42\right]_{192} & \cup & \left[90\right]_{96} & \cup & \left[6\right]_{24} & \cup & \left[3\right]_{6}   \\
         \left[1\right]_3 & = & \left[4\right]_{12} & \cup & \left[34\right]_{48} & \cup & \left[10\right]_{192} & \cup & \left[106\right]_{192} & \cup & \left[58\right]_{96} & \cup & \left[22\right]_{24} & \cup & \left[1\right]_{6}   \\
         \left[2\right]_3 & = & \left[8\right]_{12} & \cup & \left[2\right]_{48} & \cup & \left[74\right]_{192} & \cup & \left[170\right]_{192} & \cup & \left[26\right]_{96} & \cup & \left[14\right]_{24} & \cup & \left[5\right]_{6}   \\
      };
    \end{tikzpicture}
    \end{equation*}
\begin{theorem}
The function $T_R^*:\mathbb{Z}\rightarrow \mathbb{Z}$ maps the congruence classes of $3$ as shown in the diagram below:
\begin{subequations}
    \begin{equation}\label{Eqn03}
    \begin{tikzpicture}[baseline=(current  bounding  box.center)]
      \matrix (m) [matrix of math nodes,row sep=2em,column sep=-5pt]
      {
         [0]_3 & = & \left[0\right]_{12} & \cup & \left[18\right]_{48} & \cup & \left[138\right]_{192} & \cup & \left[42\right]_{192} & \cup & \left[90\right]_{96} & \cup & \left[6\right]_{24} & \cup & \left[3\right]_{6}   \\
         {} & {} & \left[0\right]_{9} & \cup & \left[3\right]_{9} & \cup & \left[6\right]_{9} & \cup & \left[0\right]_{3} & \cup & \left[8\right]_{9} & \cup & \left[2\right]_{9} & \cup & \left[5\right]_{9}   \\
      };
      \path[-stealth]
        (m-1-3) edge node [left] {$T_R$} (m-2-3)
        (m-1-5) edge node [left] {$T_R^{\,2}$} (m-2-5)
        (m-1-7) edge node [left] {$T_R^{\,3}$} (m-2-7)
        (m-1-9) edge node [left] {$T_R^{\,3}$} (m-2-9)
        (m-1-11) edge node [left] {$T_R^{\,3}$} (m-2-11)
        (m-1-13) edge node [left] {$T_R^{\,2}$} (m-2-13)
        (m-1-15) edge node [left] {$T_R$} (m-2-15);
      \matrix (m) [matrix of math nodes, row sep=2em, column sep=-4pt]
      {
        {} & {} & {} & {} & {} \\[36pt]
        \hspace{0.45in}& \underbrace{\hspace{1.33in}} & \hspace{0.85in} & \underbrace{\hspace{1.15in}} & {} \\
      };
      \matrix (m) [matrix of math nodes, row sep=2em, column sep=-4pt]
      {
        {} & {} & {} & {} & {} \\[63pt]
        \hspace{0.70in} & \left[0\right]_3 & \hspace{1.85in} & \left[2\right]_3 & \hspace{0.1in} \\
      };
    \end{tikzpicture}
    \end{equation}
    \begin{equation}\label{Eqn13}
    \begin{tikzpicture}[baseline=(current  bounding  box.center)]
      \matrix (m) [matrix of math nodes,row sep=2em,column sep=-5pt]
      {
         [1]_3 & = & \left[4\right]_{12} & \cup & \left[34\right]_{48} & \cup & \left[10\right]_{192} & \cup & \left[106\right]_{192} & \cup & \left[58\right]_{96} & \cup & \left[22\right]_{24} & \cup & \left[1\right]_{6}   \\
         {} & {} & \left[3\right]_{9} & \cup & \left[6\right]_{9} & \cup & \left[0\right]_{9} & \cup & \left[1\right]_{3} & \cup & \left[5\right]_{9} & \cup & \left[8\right]_{9} & \cup & \left[2\right]_{9}   \\
      };
      \path[-stealth]
        (m-1-3) edge node [left] {$T_R$} (m-2-3)
        (m-1-5) edge node [left] {$T_R^{\,2}$} (m-2-5)
        (m-1-7) edge node [left] {$T_R^{\,3}$} (m-2-7)
        (m-1-9) edge node [left] {$T_R^{\,3}$} (m-2-9)
        (m-1-11) edge node [left] {$T_R^{\,3}$} (m-2-11)
        (m-1-13) edge node [left] {$T_R^{\,2}$} (m-2-13)
        (m-1-15) edge node [left] {$T_R$} (m-2-15);
      \matrix (m) [matrix of math nodes, row sep=2em, column sep=-4pt]
      {
       {} & {} & {} & {} & {} \\[36pt]
        \hspace{0.40in}& \underbrace{\hspace{1.30in}} & \hspace{0.90in} & \underbrace{\hspace{1.20in}} & {} \\
      };
      \matrix (m) [matrix of math nodes, row sep=2em, column sep=-4pt]
      {
       {} & {} & {} & {} & {} \\[63pt]
        \hspace{0.53in} & \left[0\right]_3 & \hspace{1.92in} & \left[2\right]_3 & {} \\
      };
    \end{tikzpicture}
    \end{equation}
    \begin{equation}\label{Eqn23}
    \begin{tikzpicture}[baseline=(current  bounding  box.center)]
      \matrix (m) [matrix of math nodes,row sep=2em,column sep=-5pt]
      {
         [2]_3 & = & \left[8\right]_{12} & \cup & \left[2\right]_{48} & \cup & \left[74\right]_{192} & \cup & \left[170\right]_{192} & \cup & \left[26\right]_{96} & \cup & \left[14\right]_{24} & \cup & \left[5\right]_{6}   \\
         {} & {} & \left[6\right]_{9} & \cup & \left[0\right]_{9} & \cup & \left[3\right]_{9} & \cup & \left[2\right]_{3} & \cup & \left[2\right]_{9} & \cup & \left[5\right]_{9} & \cup & \left[8\right]_{9}   \\
      };
      \path[-stealth]
        (m-1-3) edge node [left] {$T_R$} (m-2-3)
        (m-1-5) edge node [left] {$T_R^{\,2}$} (m-2-5)
        (m-1-7) edge node [left] {$T_R^{\,3}$} (m-2-7)
        (m-1-9) edge node [left] {$T_R^{\,3}$} (m-2-9)
        (m-1-11) edge node [left] {$T_R^{\,3}$} (m-2-11)
        (m-1-13) edge node [left] {$T_R^{\,2}$} (m-2-13)
        (m-1-15) edge node [left] {$T_R$} (m-2-15);
      \matrix (m) [matrix of math nodes, row sep=2em, column sep=-4pt]
      {
        {} & {} & {} & {} & {} \\[36pt]
        \hspace{0.45in}& \underbrace{\hspace{1.20in}} & \hspace{0.90in} & \underbrace{\hspace{1.20in}} & {} \\
     };
      \matrix (m) [matrix of math nodes, row sep=2em, column sep=-4pt]
      {
        {} & {} & {} & {} & {} \\[63pt]
        \hspace{0.62in} & \left[0\right]_3 & \hspace{1.88in} & \left[2\right]_3 & \hspace{0.1in} \\
      };
    \end{tikzpicture}
    \end{equation}
\end{subequations}
\end{theorem}

\begin{proof}
The proof follows by verification of the individual mappings using the definition of $T_R$ from Equation (\ref{EqnTR}).  As an example, to verify the third entry in (\ref{Eqn23}), we have
\begin{align*}
T_R^3(192j+74) &= T_R^2(48j+18) \\
    &= T_R(12j+4) \\
    &= 9j+3.
\end{align*}
\end{proof}

While we can express $T_R^*$ as a piecewise function\footnote{$T_R^*(n)=(3n+1)/2$ if $n\equiv 1$ (mod $2$), $3n/4$ if $n\equiv 0$ (mod $4$), $(3n-2)/8$, if $n\equiv 6$ (mod $8$), $(3n-6)/16$ if $n\equiv 2$ (mod $16$), $(3n-14)/32$ if $n\equiv 26$ (mod $32$), $(3n-30)/64$ if $n\equiv 10$ (mod $64$), $(n-42)/64$ if $n\equiv 42$ (mod $64$).}, the relationships displayed in (\ref{Eqn03})--(\ref{Eqn23}) reveal interesting symmetries that are not apparent from a formula. For example, the mapping sends the first three subsets of each partition to the three sets $\left[0\right]_9, \left[3\right]_9, \text{\ and\ }\left[6\right]_9$,  permuted cyclically.  A similar pattern appears for the last three subsets in each group.

Mappings (\ref{Eqn03}) and (\ref{Eqn23}) also reflect a striking symmetry between the sets $\left[0\right]_3$ and $\left[2\right]_3$ that does not involve the set $\left[1\right]_3$.  In fact, the next theorem shows that numbers congruent to $\modd 1 3$ are transient in $T_R^*$-trajectories in the same way that numbers congruent to $\modd 0 3$ are transient in $T$-trajectories, as expressed in Proposition \ref{PropA}.

\begin{theorem}\label{ThmTRStar}
Every $T_R^*$-trajectory consists of finitely many numbers congruent to $\modd 1 3$ followed only by numbers congruent to either $0$ or $\modd 2 3$.
\end{theorem}
\begin{proof}
By (\ref{Eqn03}) and (\ref{Eqn23}), $T_R^*$ maps a number in $[0]_3$ or  $[2]_3$ to a number in one of those same two sets. So we need only show that every number in $[1]_3$ eventually iterates out of $[1]_3$.  From (\ref{Eqn13}), if $n\in [1]_3$ and $T_R^*(n)\in [1]_3$, then  $n=192j+106$ for some integer $j$.  In that case, $T_R^*(n)=T_R^3(n)=(n-42)/64$, from which $|T_R^*(n)|\leq|\frac{43n}{64}|<|n|$. Any sequence of $T_R^*$-iterates contained in $[1]_3$ therefore has decreasing magnitudes and so must have an element $m$ of least absolute value.  Then $T_R^*(m)$ is in $[0]_3$ or $[2]_3$.
\end{proof}

\section{Reversion to \texorpdfstring{$T$}{T}-trajectories}\label{SecReversion}

We have shown that the refined $3x+1$ conjecture (Conjecture \ref{ConjR}) and the special properties of $T_R$ reveal patterns that are not easily discerned from the standard formulation of the conjecture.  While it seems advantageous to study $T_R$ instead of $T$, it may also be of interest to transform results about $T_R$-trajectories back into the original setting of $T$-trajectories.  We do so in this section, beginning with a new characterization of $T$-trajectories.

\begin{theorem}\label{ThmB}
The $T$-trajectory of every nonzero integer $n$ has the following structure: finitely many numbers congruent to $\modd 0 3$, followed by numbers congruent to $2$ or $\modd 8 9$ except for isolated numbers (i.e., no two in succession) that are congruent to $\modd 1 3$ or isolated numbers congruent to $\modd 5 9$.
\end{theorem}

\begin{proof}
 From Proposition \ref{PropA}, in the $T$-trajectory for a nonzero $n$, numbers congruent to $\modd 0 3$ occur only as initial values, and there are finitely many such numbers. From Proposition  \ref{PropB}, numbers congruent to $\modd 1 3$ are isolated in the subsequent portion of the trajectory. The remaining numbers are those congruent to $\modd 2 3$.  We then need only show that numbers congruent to $\modd 5 9$ are isolated within that set.  But by Proposition \ref{PropC}, numbers of the form $3j+1$ never appear consecutively in a $T_R$-trajectory, which implies that numbers of the form $S(3j+1)=9j+5$ never appear consecutively in a $T$-trajectory. This says that numbers congruent to $\modd 5 9$ are isolated in the $T$-trajectory.
\end{proof}

\begin{example}
The structure described in Theorem \ref{ThmB} is illustrated in Figure \ref{FigClasses}, for the trajectory of $156$.
\end{example}
\begin{figure}[H]
\begin{center}
    \scalebox{0.65}{\includegraphics{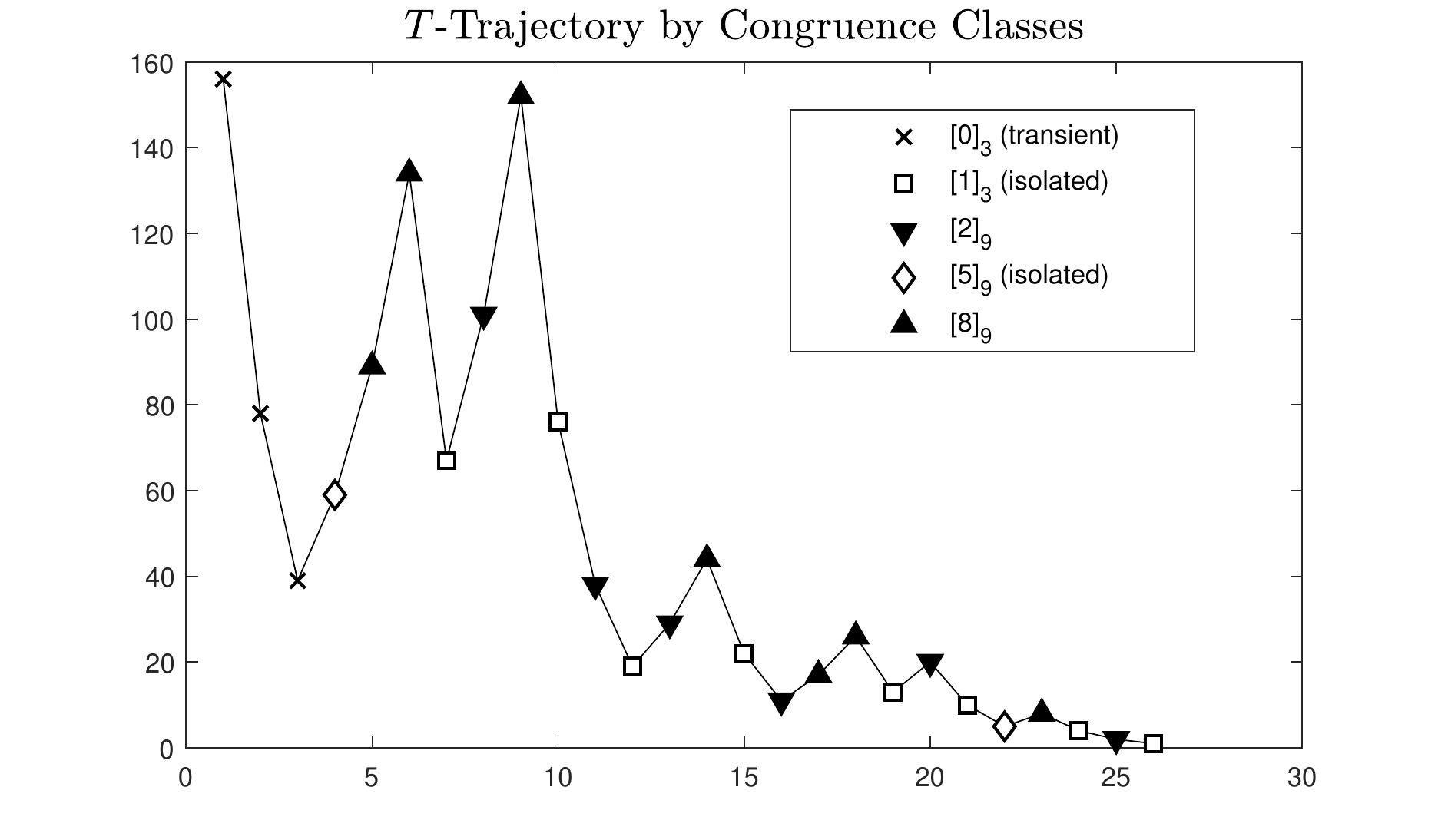}}
\caption{The $T$-trajectory of $156$, showing the properties stated in Theorem \protect\ref{ThmB}.}
\label{FigClasses}
\end{center}
\end{figure}

We can demonstrate even more clearly the distinctive role in $T$-trajectories of numbers congruent to $2$ or $\modd 8 9$ by the following result for an accelerated iteration whose long-term behavior involves only those numbers.
As in Equation (\ref{EqnZpart}), we use a special partition, this time for the set  $[2]_3$.  We also express this partition with residues of least absolute value in order to more clearly show the pattern.
\begin{align*}
[2]_3
    &= [5]_6 \cup [2]_{6} \\
    &= [5]_6 \cup [2]_{12} \cup [8]_{12} \\
    &= [5]_6 \cup [2]_{12} \cup [20]_{24} \cup [8]_{24} \\
    &= [5]_6 \cup [2]_{12} \cup [20]_{24} \cup [8]_{48} \cup [32]_{48} \\
    &= [5]_6 \cup [2]_{12} \cup [20]_{24} \cup [8]_{48} \cup [80]_{96} \cup [32]_{96}  \\
    &= [5]_6 \cup [2]_{12} \cup [20]_{24} \cup [8]_{48} \cup [80]_{96} \cup [32]_{192} \cup [128]_{192} \\
    &= [-1]_6 \cup [2]_{12} \cup [-4]_{24} \cup [8]_{48} \cup [-16]_{96} \cup [32]_{192} \cup [-64]_{192}
\end{align*}

\begin{theorem}\label{ThmTstar}
Define $T^*:\left[2\right]_3 \rightarrow \mathbb{Z}$ by
\begin{displaymath}
T^*(n)=
\begin{dcases}
  T(n)=\dfrac{3n+1}{2}, & \text{if $n\equiv -1 \pmod 6$}\\
  T^{\,2}(n)=\dfrac{3n+2}{4}, & \text{if $n \equiv 2 \pmod {12}$} \\
  T^{\,3}(n)=\dfrac{3n+4}{8}, & \text{if $n\equiv -4 \pmod {24}$} \\
  T^{\,4}(n)=\dfrac{3n+8}{16}, & \text{if $n\equiv 8 \pmod {48}$} \\
  T^{\,5}(n)=\dfrac{3n+16}{32}, & \text{if $n\equiv -16 \pmod {96}$} \\
  T^{\,6}(n)=\dfrac{3n+32}{64}, & \text{if $n\equiv 32 \pmod {192}$} \\
  T^{\,6}(n)=\dfrac{n}{64}, & \text{if $n\equiv -64 \pmod {192}$}.
\end{dcases}
\end{displaymath}
Then
\begin{enumerate}[label={(\roman*)}]
\item
the range of $T^*$ is $\left[2\right]_3$, \label{ThmTstar1}
\item
every $T^*$-trajectory consists of finitely many numbers congruent to $\modd 5 9$ followed only by numbers congruent to either $2$ or $\modd 8 9$. \label{ThmTstar2}
\end{enumerate}
\end{theorem}

\begin{proof}
By applying $S(x)=3x+2$ to the sets shown in (\ref{Eqn03})--(\ref{Eqn23}) and making use of Equations (\ref{EqnSofT}) and (\ref{EqnF}), we obtain the following mapping of numbers congruent to $2$ (mod $3$), listed by congruence class (mod $9$). The mappings shown in (\ref{Eqn29})--(\ref{Eqn89}) coincide with the definition of $T^*$, as may be verified directly from the definition of $T$.  Result \emph{\ref{ThmTstar1}} is then a consequence of  (\ref{Eqn29})--(\ref{Eqn89}), and \emph{\ref{ThmTstar2}} follows by applying $S$ to the numbers described in Theorem \ref{ThmTRStar}.
\begin{subequations}
    \begin{equation}\label{Eqn29}
    \vspace{-24pt}
    \begin{tikzpicture}[baseline=(current  bounding  box.center)]
          \matrix (m) [matrix of math nodes,row sep=2em,column sep=-5pt]
          {
             [2]_9 & = & \left[2\right]_{36} & \cup & \left[56\right]_{144} & \cup & \left[416\right]_{576} & \cup & \left[128\right]_{576} & \cup & \left[272\right]_{288} & \cup & \left[20\right]_{72} & \cup & \left[11\right]_{18}   \\
             {} & {} & \left[2\right]_{27} & \cup & \left[11\right]_{27} & \cup & \left[20\right]_{27} & \cup & \left[2\right]_{9} & \cup & \left[26\right]_{27} & \cup & \left[8\right]_{27} & \cup & \left[17\right]_{27}   \\
          };
          \path[-stealth]
            (m-1-3) edge node [left] {$T^{\,2}$} (m-2-3)
            (m-1-5) edge node [left] {$T^{\,4}$} (m-2-5)
            (m-1-7) edge node [left] {$T^{\,6}$} (m-2-7)
            (m-1-9) edge node [left] {$T^{\,6}$} (m-2-9)
            (m-1-11) edge node [left] {$T^{\,5}$} (m-2-11)
            (m-1-13) edge node [left] {$T^{\,3}$} (m-2-13)
            (m-1-15) edge node [left] {$T$} (m-2-15);
          \matrix (m) [matrix of math nodes, row sep=2em, column sep=-4pt]
          {
            {} & {} & {} & {} & {} \\[36pt]
            \hspace{0.40in}& \underbrace{\hspace{1.45in}} & \hspace{0.88in} & \underbrace{\hspace{1.48in}} & {} \\
          };
          \matrix (m) [matrix of math nodes, row sep=2em, column sep=-4pt]
          {
            {} & {} & {} & {} & {} \\[63pt]
            \hspace{0.55in} & \left[2\right]_9 & \hspace{2.15in} & \left[8\right]_9 & \hspace{0.1in} \\
          };
    \end{tikzpicture}
    \end{equation}
        \begin{equation}\label{Eqn59}
    \vspace{-24pt}
    \begin{tikzpicture}[baseline=(current  bounding  box.center)]
          \matrix (m) [matrix of math nodes,row sep=2em,column sep=-5pt]
          {
             [5]_9 & = & \left[14\right]_{36} & \cup & \left[104\right]_{144} & \cup & \left[32\right]_{576} & \cup & \left[320\right]_{576} & \cup & \left[176\right]_{288} & \cup & \left[68\right]_{72} & \cup & \left[5\right]_{18}   \\
             {} & {} & \left[11\right]_{27} & \cup & \left[20\right]_{27} & \cup & \left[2\right]_{27} & \cup & \left[5\right]_{9} & \cup & \left[17\right]_{27} & \cup & \left[26\right]_{27} & \cup & \left[8\right]_{27}   \\
          };
          \path[-stealth]
            (m-1-3) edge node [left] {$T^{\,2}$} (m-2-3)
            (m-1-5) edge node [left] {$T^{\,4}$} (m-2-5)
            (m-1-7) edge node [left] {$T^{\,6}$} (m-2-7)
            (m-1-9) edge node [left] {$T^{\,6}$} (m-2-9)
            (m-1-11) edge node [left] {$T^{\,5}$} (m-2-11)
            (m-1-13) edge node [left] {$T^{\,3}$} (m-2-13)
            (m-1-15) edge node [left] {$T$} (m-2-15);
          \matrix (m) [matrix of math nodes, row sep=2em, column sep=-4pt]
          {
            {} & {} & {} & {} & {} \\[36pt]
            \hspace{0.38in}& \underbrace{\hspace{1.55in}} & \hspace{0.85in} & \underbrace{\hspace{1.40in}} & {} \\
          };
          \matrix (m) [matrix of math nodes, row sep=2em, column sep=-4pt]
          {
            {} & {} & {} & {} & {} \\[63pt]
            \hspace{0.65in} & \left[2\right]_9 & \hspace{2.10in} & \left[8\right]_9 & \hspace{0.1in} \\
          };
    \end{tikzpicture}
    \end{equation}
\begin{equation}\label{Eqn89}
        \begin{tikzpicture}[baseline=(current  bounding  box.center)]
          \matrix (m) [matrix of math nodes,row sep=2em,column sep=-5pt]
          {
             [8]_9 & = & \left[26\right]_{36} & \cup & \left[8\right]_{144} & \cup & \left[224\right]_{576} & \cup & \left[512\right]_{576} & \cup & \left[80\right]_{288} & \cup & \left[44\right]_{72} & \cup & \left[17\right]_{18}   \\
             {} & {} & \left[20\right]_{27} & \cup & \left[2\right]_{27} & \cup & \left[11\right]_{27} & \cup & \left[8\right]_{9} & \cup & \left[8\right]_{27} & \cup & \left[17\right]_{27} & \cup & \left[26\right]_{27}   \\
          };
          \path[-stealth]
            (m-1-3) edge node [left] {$T^{\,2}$} (m-2-3)
            (m-1-5) edge node [left] {$T^{\,4}$} (m-2-5)
            (m-1-7) edge node [left] {$T^{\,6}$} (m-2-7)
            (m-1-9) edge node [left] {$T^{\,6}$} (m-2-9)
            (m-1-11) edge node [left] {$T^{\,5}$} (m-2-11)
            (m-1-13) edge node [left] {$T^{\,3}$} (m-2-13)
            (m-1-15) edge node [left] {$T$} (m-2-15);
          \matrix (m) [matrix of math nodes, row sep=2em, column sep=-4pt]
          {
            {} & {} & {} & {} & {} \\[36pt]
            \hspace{0.40in}& \underbrace{\hspace{1.45in}} & \hspace{0.85in} & \underbrace{\hspace{1.40in}} & {} \\
          };
          \matrix (m) [matrix of math nodes, row sep=2em, column sep=-4pt]
          {
            {} & {} & {} & {} & {} \\[63pt]
            \hspace{0.60in} & \left[2\right]_9 & \hspace{2.05in} & \left[8\right]_9 & \hspace{0.1in} \\
          };
        \end{tikzpicture}
    \end{equation}
\end{subequations}
\end{proof}

\begin{example}
To illustrate Theorem (\ref{ThmTstar}), we consider again the $T$-trajectory of $156$, which was shown in Figure \ref{FigClasses}.  That trajectory is
\[
\left(156, 78, 39, 59, 89, 134, 67, 101, 152, 76, 38, 19, 29, 44, 22, 11, 17, 26, 13, 20, 10, 5, 8, 4, 2, 1, \ldots \right).
\]
The first three terms are elements of the transient set $[0]_3$, so the $T^*$-trajectory begins with 59 and is as follows:
\[
\left(59, 89, 134, 101, 152, 29, 44, 17, 26, 20, 8, 2, \ldots\right).
\]
Note that the first term is congruent to $5$ (mod $9$) and all subsequent terms are congruent to either $2$ or $8$ (mod $9$). However, not all terms of a $T$-trajectory that are congruent to $2$ or $8$ (mod $9$) appear in the accelerated $T^*$-trajectory.  In this example, $38$ and $11$ are two such numbers that are bypassed in the iteration by $T^*$.
\end{example}

\section{Additional computational results}\label{SecComp}

It it common in the literature (see, e.g., Lagarias \cite{Lag0}) to let $\sigma_\infty(n)$ denote the total stopping time of $n$, which is defined to be the least $k$ for which $T^k(n)=1$, if such a $k$ exists, and $\infty$ otherwise. We denote this here by $\sigma_T(n)$ and define the analogous quantity $\sigma_{T_R}(n)$ to be the least $k$ for which $T_R^k(n)=0$, if such $k$ exists, and $\infty$ otherwise.

Table \ref{TableA} compares $\sigma_T(n)$ and $\sigma_{T_R}(S^{-1}(n))$  for numbers $n$ of increasing trajectory length and with $n\equiv \modd{2} {3}$.   The difference in the two quantities is equal to the number of terms congruent to $1$ (mod $3$) in the $T$-trajectory of $n$.

We can obtain a heuristic estimate for the ratio of these two values by noting that a number  congruent to $2$ (mod $3$) is either of the form $6j+2$, which iterates to $3j+1$, or the form $6j+5$, which iterates to $9j+8$.  So the probability is $1/2$ that a randomly chosen number congruent to $2$ (mod $3$) iterates to a number congruent to $1$ (mod $3$).  By Proposition \ref{PropB}, numbers congruent to $1$ (mod $3$) never appear in succession, so if $P_k$ denotes the probability that the $k^{th}$ term in a $T$-trajectory is congruent to $1$ (mod $3$), we have $P_k = P_{k-1}\cdot 0 + \left(1-P_{k-1}\right)\cdot \frac{1}{2} = \frac{1}{2} \left(1-P_{k-1}\right)$.
By iterating, $P_k=\frac{1}{2}-\frac{1}{4}+\frac{1}{8}-\ldots$, which approaches $\frac{1}{3}$ as $k\rightarrow\infty$.  So assuming sufficient mixing of numbers to justify treating them as randomly congruent to $1$ or $2$ (mod $3$), we expect long $T$-trajectories to have about $\frac{1}{3}$ of their terms congruent to $1$ (mod $3$) and $\frac{2}{3}$ of the terms congruent to $2$ (mod $3$). The ratio of $\sigma_{T_R}(S^{-1}(n))$ to $\sigma_T(n))$ should then be, on average, about $\frac{2}{3}$.  This result is approximately reflected by the values in Table \ref{TableA}.

\begin{table}[H]
\begin{center}
\begin{tabular}{rrrc} \toprule
      \multicolumn{1}{c}{$n$}
    & \multicolumn{1}{c}{$\sigma_T(n)$\hspace{0.6cm}}
    & \multicolumn{1}{c}{$\sigma_{T_R}(S^{-1}(n))$}
    & \multicolumn{1}{c}{$\frac{\sigma_{T_R}(S^{-1}(n))}{\sigma_T(n)}$}\\
        \midrule
        $41$\hspace{0.5cm}   & $70$~~ & $49$\hspace{0.7cm} & $0.700$ \\
        $1307$\hspace{0.5cm}  & $113$~~  &  $79$\hspace{0.7cm} & $0.699$ \\
        $13886$\hspace{0.5cm} & $164$~~  &  $114$\hspace{0.7cm} & $0.695$  \\
        $115547$\hspace{0.5cm} & $221$~~ & $157$\hspace{0.7cm} & $0.710$ \\
        $1256699$\hspace{0.5cm} & $329$~~ & $233$\hspace{0.7cm} & $0.708$ \\ \bottomrule
\end{tabular}
\caption{Total stopping times under $T$ and $T_R$.}{}
\label{TableA}
\end{center}
\end{table}

It is also interesting to examine the distribution of numbers in a $T_R$-trajectory among the congruence classes of $3$, since Proposition \ref{PropC} suggests that $T_R$-trajectories should have relatively few terms congruent to $1$ (mod $3$). Table \ref{TableB} provides examples of distributions of $T_R$-iterates among the three congruence classes.

\begin{table}[H]
\begin{center}
\begin{tabular}{rrrrr} \toprule
      \multicolumn{1}{c}{$n$}
    & \multicolumn{1}{c}{$\sigma_{T_R}(S^{-1}(n))$}
    & \multicolumn{1}{c}{$[0]_3$ count}
    & \multicolumn{1}{c}{$[1]_3$ count}
    & \multicolumn{1}{c}{$[2]_3$ count} \\
    \midrule
    $41$ & $49$\hspace{0.7cm} & $16$\hspace{0.5cm} & $5$\hspace{0.5cm} & $28$\hspace{0.5cm}  \\
    $1307$ & $79$\hspace{0.7cm} & $27$\hspace{0.5cm} & $7$\hspace{0.5cm} & $45$\hspace{0.5cm}  \\
    $13886$ & $114$\hspace{0.7cm} & $42$\hspace{0.5cm} & $7$\hspace{0.5cm} & $65$\hspace{0.5cm}  \\
    $115547$ & $157$\hspace{0.7cm} & $54$\hspace{0.5cm} & $9$\hspace{0.5cm} & $94$\hspace{0.5cm} \\
    $1256699$ & $233$\hspace{0.7cm} & $79$\hspace{0.5cm} & $16$\hspace{0.5cm} & $138$\hspace{0.5cm} \\ \bottomrule
\end{tabular}
\caption{Distribution of terms in $T_R$-trajectories by congruence classes (mod $3$).}{}
\label{TableB}
\end{center}
\end{table}

\section{Summary, and new questions}\label{SecSummary}

We have presented a streamlined formulation of the $3x+1$ conjecture that eliminates certain extraneous features of trajectories and reveals new structure.  It is the author's hope that the results presented here will help bring to light other  features of the $3x+1$ problem.  Some possible areas for subsequent investigation are

\begin{description}
\item[Implications of the $T_R$ partition mapping]
Do Equations (\ref{EqnPartitions}) for the $T_R$ partition mapping have further implications for the refined $3x+1$ conjecture?
\item[Asymptotic distributions of congruence classes]
The heuristic argument given above for the asymptotic proportion of numbers in a $T$-trajectory that are congruent to $2$ (mod $3$) gives a value of $\frac{2}{3}$, whereas experimental values seem to be consistently closer to $0.70$ or $0.71$.  Is there an explanation for this?  In a similar vein, can anything be said about asymptotic distributions of the congruence classes within $T_R$-trajectories as illustrated in Table \ref{TableB}?
\item[Parities relative to $0$ or $2$ (mod $3$)]
Mappings (\ref{Eqn03}) and (\ref{Eqn23}) express iterates of $T_R$ in terms only of numbers congruent to $0$ or $2$ (mod $3$), with intriguing symmetry between these two sets. Can more insight be gained by using $\{0,1\}$ parity vectors to represent membership in these sets in the same way that even-odd parities have been useful in analyzing $T$-trajectories?   (See, for example, Lagarias \cite{Lag0} and Rozier \cite{Roz}.)
\item[Extensions to $\mathbb{R}$ or $\mathbb{C}$]
Chamberland \cite{Cha} obtained interesting results from an extension of $T$ to a function on $\mathbb{R}$.  Others (e.g.,\ Letherman, et al.\ \cite{Let}) have studied extensions to $\mathbb{C}$. Can these methods be usefully applied to similar extensions of $T_R$, for example, by the following analogue of Chamberland's function that interpolates $T_R$ at the integers?

\[
f(x) =\frac{3x+1}{2} + \frac{x+1}{4}\cos\left(\frac{\pi x}{2}\right) - \frac{4x+3}{4}\cos^2\left(\frac{\pi x}{2}\right).
\]
\end{description}

\medskip

\noindent MSC2010: 11B83

\end{document}